\documentclass[11pt,a4paper]{amsart}
\usepackage{
amssymb,
amsmath}

\usepackage{times}
\usepackage[bookmarks,colorlinks,pagebackref]{hyperref} 
\usepackage[all,cmtip,2cell]{xy}
\usepackage{color}

\numberwithin{equation}{section}
\newtheorem{theorem}{Theorem}[section]
\newtheorem{lemma}[theorem]{Lemma}
\newtheorem{proposition}[theorem]{Proposition}
\newtheorem{corollary}[theorem]{Corollary}

\theoremstyle{definition}
\newtheorem{definition}[theorem]{Definition}
\newtheorem{example}[theorem]{Example}
\newtheorem{problem}[theorem]{Problem}
\newtheorem{remark}[theorem]{Remark}

\newtheorem{notation}[theorem]{Notation}

\newcommand\Hom{\operatorname{Hom}}

\newcommand\Ext{\operatorname{Ext}}

\newcommand{\au}{\underline a}

\author[Khadam]{M. Azeem Khadam}
\address{Abdus Salam School of Mathematical Sciences, GCU, Lahore Pakistan}
\email{azeemkhadam@gmail.com}

\author[Schenzel]{Peter Schenzel}
\address{Martin-Luther-Universit\"at Halle-Wittenberg,
Institut f\"ur Informatik, D --- 06 099 Halle (Saale), Germany}
\email{peter.schenzel@informatik.uni-halle.de}
\thanks{The first named author is grateful to DAAD and HEC, Pakistan for the support of this research under 
	grant number 91524811 and 112-21480-2PS1-015 (50021731) respectively.}

\title[Multiplicities]{About multiplicities and applications to Bezout numbers}

\begin{document}

	\begin{abstract} 
		Let $(A,\mathfrak{m},\Bbbk)$ denote a local Noetherian ring and $\mathfrak{q}$ an ideal 
		such that $\ell_A(M/\mathfrak{q}M) < \infty$ for a finitely generated $A$-module $M$. 
		Let $\au = a_1,\ldots,a_d$ denote a system of parameters of $M$ such that $a_i \in \mathfrak{q}^{c_i} 
		\setminus \mathfrak{q}^{c_i+1}$ for $i=1,\ldots,d$. It follows that $ \chi := e_0(\au;M) - c \cdot e_0(\mathfrak{q};M) \geq 0$, where $c = c_1\cdot \ldots \cdot c_d$. The main results of the report are a discussion when $\chi = 0$ resp. 
		to describe the value of $\chi$ in some particular cases. Applications concern results on the 
		multiplicity $e_0(\au;M)$ and applications to Bezout numbers.  
	\end{abstract}
	

\maketitle
\begin{center}
	{\it {Dedicated to Winfried Bruns on the occasion of his 70th birthday.}}	
\end{center}

\section{Introduction} 
Let $(A,\mathfrak{m},\Bbbk)$ denote a local ring. Let $\mathfrak{q} \subset A$ be an $\mathfrak{m}$-primary ideal 
and $\au = a_1,\ldots,a_d$ be a system of parameters in $A$ such that $a_i \in \mathfrak{q}^{c_i}, 
i = 1,\ldots,d,$ with $c_i > 0$. The main interest in the present report is a comparison 
of the multiplicities $e_0(\au;A)$ and $e_0(\mathfrak{q};A)$. 

Let $M$ be a finitely generated 
$A$-module. Note that for an ideal $\mathfrak{q} \subset A$ such that the length $\ell_A(M/\mathfrak{q}M)$ is finite, the multiplicity $e_0(\mathfrak{q};M)$ is defined as the leading term of the Hilbert-Samuel polynomial 
\[
\ell_A(M/\mathfrak{q}^{n+1}M) = \sum_{i=0}^{d} e_i(\mathfrak{q};M) \binom{n+d-i}{d-i} \, \text{ for } n \gg 0, \text{ with } d = \dim_A M
\] 
(see for instance \cite{ZS}, \cite{Ma} for all the details or more general \cite{RV} for 
generalizations to filtered modules). With the previous assumption, clearly $e_0(\au;A) \geq 
e_0(\mathfrak{q};A)$. We will discuss this relation in more detail. First let us recall some known results:
\begin{itemize}
	\item[(a)] For $\mathfrak{q} = \au A$ we get $e_0(\au^{\underline{n}};A) = n \cdot e_0(\au;A)$, where 
	$\au^{\underline{n}} = a_1^{n_1}, \ldots, a_d^{n_d}$ and $n = n_1 \cdot \ldots \cdot n_d$ for some 
	$(n_1, \ldots,n_d) \in \mathbb{N}^d$.
	\item[(b)] If $\au$ is a minimal reduction of $\mathfrak{q}$, then $e_0(\au;A) = e_0(\mathfrak{q};A)$. 
	The converse is true (see \cite{Re1}) provided $A$ is formally equidimensional. 
	\item[(c)]  It follows that $e_0(\au;A) \geq c \cdot e_0(\mathfrak{q};A)$, where $c = c_1\cdot \ldots \cdot c_d$ 
	with the above notation. 
\end{itemize}

For the proof of (a) we refer to \cite{AB}. The first part of (b) is well-known, while the 
converse is an outstanding result of Rees (see \cite{Re1}). The claim in (c) is easy to prove (see 
\cite{BS} for details). The main goal of the present report is a discussion of the 
difference $\chi := e_0(\au;A) - c \cdot e_0(\mathfrak{q};A) \geq 0$ of (c) in various 
situations, its vanishing  resp. a simplified proof of some known results. 

The importance of the understanding of $\chi$ has to do with Bezout's Theorem in the plane. 
Let $C = V(F), D = V(G) \subset \mathbb{P}^2_{\Bbbk}, \Bbbk$ an algebraically closed field, be two curves in the projective plane without a common component. Then 
\[
\sum_{P \in C\cap D} \mu(P;C,D) = \deg C \cdot \deg D,
\]
where $\mu(P;C,D)$ denotes the local intersection multiplicity of $P$ in $C \cap D$. In the case 
of $P$ is the origin, it follows that $\mu(P;C,D) = e_0(f,g;A)$, where $A = \Bbbk[x,y]_{(x,y)}$ and $f,g$ denote 
the equations in $A$. Since $C,D$ have no component in common, $\{f,g\}$ forms as system of parameters in $A$. Then 
\[
e_0(f,g;A) \geq c\cdot d \cdot e_0(\mathfrak{m};A) = c\cdot d,
\]
since $e_0(\mathfrak{m};A) = 1$. Here $c, d$ denote the initial degree of $f,g$ respectively. This 
estimate is well-known (see for instance \cite{BK} or \cite{Fi}) proved by resultants or Puiseux expansions. 
Moreover equality holds if and only if $C,D$ intersect transversally in the origin. In other words, $f^{\star},g^{\star}$, 
the initial forms of $f,g$ in the form ring $G_A(\mathfrak{m}) \cong \Bbbk[X,Y]$ 
are a homogeneous system of parameters. 

Here we shall provide another argument with extensions to arbitrary local rings.  Let $A$ denote a local 
ring with $\au = a_1,\ldots,a_d \in A$ a system of parameters. Put $c = c_1\cdot \ldots \cdot c_d$ 
for $a_i \in \mathfrak{m}^{c_i} \setminus \mathfrak{m}^{c_i +1}$.  In 
his paper (see \cite{Pr}) the author claimed that $e_0(\au;A) = c \cdot  e_0(\mathfrak{m};A)$ 
if and only if the sequence of initial elements  $a_1^{\star}, \ldots ,a_d^{\star} \in G_A(\mathfrak{m})$ 
forms a regular sequence. This is not true as the following example shows. 

\begin{example} \label{exam-1}
	Let $\Bbbk$ denote a field and $A = \Bbbk[|t^4,t^5,t^{11}|] 
	\subset \Bbbk[|t|],$ where $t$ is an indeterminate over $\Bbbk.$ Then $A$ is a one-dimensional 
	domain and therefore a Cohen-Macaulay ring with $A \simeq \Bbbk[|X,Y,Z|]/(X^4-YZ,Y^3-XZ,Z^2-X^3Y^2).$ 
	Clearly, the residue class $a = x$ of $X$ is a parameter with $a \in \mathfrak{m} \setminus 
	\mathfrak{m}^2,$ so that $c = 1.$  
	
	Furthermore, by easy calculations it follows that $e_0(a,A) = \ell_A(A/aA) = 4$ and 
	$e_0(\mathfrak{m},A) = 4.$ So, the equation $e_0(a,A) = c \cdot e_0(\mathfrak{m},A)$ holds, 
	while $G_A(\mathfrak{m}) = \Bbbk[X,Y,Z]/(XZ,YZ,Y^4,Z^2)$ is not a Cohen-Macaulay ring 
	(see \cite[Section 3]{BS} for the details). 
\end{example} 

In Section 2 we start with some preliminaries and Koszul complexes. In Section 3 we derive some 
new complexes from certain Koszul complexes important for the study of multiplicities. Euler 
characteristics are the feature of Section 4. As an application we derive a short argument 
for computing certain multiplicities as Euler characteristics of Koszul complexes (originally done 
by Auslander and Buchsbaum (see \cite{AB}) and Serre (see \cite{jpS}) by spectral sequence arguments). 
In Section 5 we study the equality $e_0(\au;M) = c_1 \cdot \ldots \cdot c_d \cdot e_0(\mathfrak{q};M)$. Under 
some additional regularity condition on the sequence of initial forms $a_1^{\star}, \ldots, a_{d-1}^{\star}$ 
in $G_A(\mathfrak{q})$ we estimate the difference $\ell_A(M/\au M) - c_1 \cdot \ldots \cdot c_d \cdot e_0(\mathfrak{q};M)$. 
As an application we get a bound of the local Bezout intersection numbers of two curves in the 
projective plane without common component.

\section{Preliminaries}
First let us fix the notations we will use in the following. For the basics on $\mathbb{N}$-graded 
structures we refer e.g. to \cite{GW}.

\begin{notation} \label{not-1}
	(A) We denote by $A$ a commutative Noetherian ring with $0 \not= 1$. 
	For an ideal we write $\mathfrak{q} \subset A$. An $A$-module is denoted by $M$. 
	Mostly we consider  $M$ as finitely generated. \\
	(B) We consider the Rees and form rings of $A$ with respect 
	to $\mathfrak{q}$ by 
	\[
	R_A(\mathfrak{q}) = \oplus_{n \geq 0} \mathfrak{q}^n \,T^n \subseteq A[T] \,
	\text{ and }\, G_A(\mathfrak{q}) = \oplus_{n \geq 0} \mathfrak{q}^n/\mathfrak{q}^{n+1}.
	\]
	Here $T$ denotes an indeterminate over $A$. Both rings are naturally $\mathbb{N}$-graded. For an $A$-module $M$ we define 
	the Rees and form modules in the corresponding way by 
	\[
	R_M(\mathfrak{q}) = \oplus_{n \geq 0} \mathfrak{q}^n M \,T^n \subseteq M[T] \,
	\text{ and } \, G_M(\mathfrak{q}) = \oplus_{n \geq 0} \mathfrak{q}^nM/\mathfrak{q}^{n+1}M.
	\]
	Note that $R_M(\mathfrak{q})$ is a graded $R_A(\mathfrak{q})$-module and 
	$G_M(\mathfrak{q})$ is a graded $G_A(\mathfrak{q})$-module. Note that $R_A(\mathfrak{q})$ 
	and $G_A(\mathfrak{q})$ are both Noetherian rings. In case $M$ is a finitely generated 
	$A$-module then $R_M(\mathfrak{q})$ resp. $G_M(\mathfrak{q})$ is finitely generated 
	over $R_A(\mathfrak{q})$ resp. $G_A(\mathfrak{q})$. \\
	(C) There are the following two short exact sequences of graded modules
	\begin{gather*}
	0 \to R_M(\mathfrak{q})_{+}[1] \to R_M(\mathfrak{q}) \to G_M(\mathfrak{q}) \to 0 \text{ and }\\
	0 \to R_M(\mathfrak{q})_{+} \to R_M(\mathfrak{q}) \to M \to 0,
	\end{gather*}
	where $R_M(\mathfrak{q})_{+} = \oplus_{n > 0} \mathfrak{q}^n M \,T^n$. \\
	(D) Let $m \in M$ and $m \in \mathfrak{q}^c M \setminus \mathfrak{q}^{c+1} M$. Then we define  
	$m^{\star} := m + \mathfrak{q}^{c+1} M \in [G_M(\mathfrak{q})]_c$. If $m \in \cap_{n \geq 1} \mathfrak{q}^n M$, 
	then we write $m^{\star} = 0$. $m^{\star}$ is called the initial element of $m$ in $G_M(\mathfrak{q})$ and $c$ is called the initial degree of $m$.
	Here $[X]_n, n \in \mathbb{Z},$ denotes the $n$-th graded component of an $\mathbb{N}$-graded module $X$.
\end{notation}

For these and related results we refer to \cite{GW} and \cite{SH}. Another feature for the 
investigations will be the use of Koszul complexes. 

\begin{remark} \label{kos-1} ({\sl Koszul complex.})
	(A) Let $\underline{a} = a_1,\ldots,a_t$ denote a system of elements of the ring $A$. The Koszul complex 
	$K_{\bullet}(\underline{a};A)$ is defined as follows: Let $F$ denote a free $A$-module with basis $e_1,\ldots,e_t$. 
	Then $K_i(\underline{a};A) = \bigwedge^i F$ for $i = 1,\ldots,t$. A basis of $K_i(\underline{a};A)$ is given by the wedge 
	products $e_{j_1} \wedge \ldots \wedge e_{j_i}$ for $1 \leq j_1 < \ldots < j_i \leq t$. The boundary 
	homomorphism $K_i(\underline{a};A) \to K_{i-1}(\au;A)$ is defined by 
	\[
	d_{j_1 \ldots j_i} :
	e_{j_1} \wedge \ldots \wedge e_{j_i} \mapsto \sum_{k=1}^{i} (-1)^{k+1} a_{j_k} e_{j_1}\wedge \ldots \wedge \widehat{e_{j_k}} 
	\wedge \ldots \wedge e_{j_i}
	\] 
	on the free generators $e_{j_1} \wedge \ldots \wedge e_{j_i}$. \\
	(B) Another way of the construction of $K_{\bullet}(\au;A)$ is inductively by the mapping cone. 
	To this end let $X$ denote a complex of $A$-modules. Let $a \in A$ denote an element of $A$. The 
	multiplication by $a$ on each $A$-module $X_i, i \in \mathbb{Z},$ induces a morphism of complexes 
	$m_a : X \to X$. We define $K_{\bullet}(a;X)$ as the mapping cone $\operatorname{Mc} (m_a)$. Then we define inductively 
	\[
	K_{\bullet}(a_1,\ldots,a_t;A) = K_{\bullet}(a_t,K_{\bullet}(a_1,\ldots,a_{t-1};A)).
	\]
	It is easily seen that 
	\[
	K_{\bullet}(\au;A) \cong K_{\bullet}(a_1;A) \otimes_A \cdots \otimes_A K_{\bullet}(a_t;A).
	\]
	Therefore it follows that $K_{\bullet}(\au;A) \cong K_{\bullet}(\au_{\sigma};A)$, where 
	$\au_{\sigma} = a_{\sigma(1)}, \ldots, a_{\sigma(t)}$ with a permutation $\sigma$ on $t$ letters. 
	For an $A$-complex $X$ we define $K_{\bullet}(\au;X) = K_{\bullet}(\au;A)\otimes_A X$. We write 
	$H_i(\au;X), i \in \mathbb{Z},$ for the $i$-th homology of $K_{\bullet}(\au;X)$. A short exact 
	sequence of $A$-complexes $0 \to X' \to X \to X'' \to 0$ induces a long exact homology sequence 
	for the Koszul homology
	\[
	\ldots \to H_i(\au;X') \to H_i(\au;X) \to H_i(\au;X'') \to H_{i-1}(\au;X') \to \ldots.
	\]
	Let $\au$ as above a system of $t$ elements in $A$ and $b \in A$. Then the mapping cone construction 
	provides a long exact homology sequence 
	\[
	\ldots \to H_i(\au;X) \to H_i(\au;X) \to H_i(\au,b;X) \to H_{i-1}(\au;X) \to H_{i-1}(\au;X) \to \ldots, 
	\]
	where the homomorphism $H_i(\au;X) \to H_i(\au;X)$ is multiplication by $(-1)^i b$. Moreover, $\au H_i(\au;X) = 0$ 
	for all $i \in \mathbb{Z}$.
\end{remark}

For the proof of the last statement, we recall the following well-known argument. 

\begin{lemma} \label{cone-1}
	Let $X$ denote a complex of $A$-modules. Let $a \in A$ denote an element.  Then $a H_i(a,X) = 0$ for 
	all $i \in \mathbb{Z}.$ 
\end{lemma}

\section{The construction of complexes}
First we fix notations for this section. As above let $A$ denote a commutative Noetherian ring 
and $\mathfrak{q} \subseteq A$. Let $\au = a_1,\ldots,a_t$ denote a system of elements of $A$. 
Suppose that $a_i \in \mathfrak{q}^{c_i}$ for some integers $c_i \in \mathbb{N}$ for $i = 1,\ldots,t$. 
Let $M$ denote a finitely generated $A$-module. We define two complexes here, see also \cite{KSch} for more detail. 

\begin{notation} \label{kos-2}
	Let $n$ denote an integer. We define a complex $K_{\bullet}(\au,\mathfrak{q},M;n)$ in the following way:
	\begin{itemize}
		\item[(a)] For $0 \leq i \leq t$ put $K_i(\au,\mathfrak{q},M;n) = \oplus_{1\leq j_1 < \ldots < j_i \leq t} 
		\mathfrak{q}^{n-c_{j_1}- \ldots- c_{j_i}}M$ and $K_i(\au,\mathfrak{q},M;n) = 0$ for $i >t$ or $i <0$.
		\item[(b)] The boundary map $K_i(\au,\mathfrak{q},M;n) \to K_{i-1}(\au,\mathfrak{q},M;n)$ is defined by maps on each of 
		the direct summands $\mathfrak{q}^{n-c_{j_1}- \ldots- c_{j_i}}M$. On $\mathfrak{q}^{n-c_{j_1}- \ldots- c_{j_i}}M$ 
		it is the map given by $d_{j_1 \ldots j_i} \otimes 1_M$ restricted to $\mathfrak{q}^{n-c_{j_1}- \ldots- c_{j_i}}M$, where $d_{j_1 \ldots j_i}$ denotes the homomorphism as defined in \ref{kos-1}.
	\end{itemize}
	It is clear that the image of the map is contained in $\oplus_{1\leq j_1 < \ldots < j_{i-1} \leq t} 
	\mathfrak{q}^{n-c_{j_1}- \ldots- c_{j_{i-1}}}M$. Clearly it is a boundary homomorphism. 
	By the construction it follows that $K_{\bullet}(\au,\mathfrak{q},M;n)$ is a sub complex of the 
	Koszul complex $K_{\bullet}(\au;M)$ for each $n \in \mathbb{N}$.
\end{notation}

Another way for the construction is the following.

\begin{remark} \label{rem-kos}
	Let $R_A(\mathfrak{q})$ and $R_M(\mathfrak{q})$ denote the Rees ring and the Rees module. 
	For $a_i, i = 1,\ldots,t,$ we consider $a_i T^{c_i} \in [R_A(\mathfrak{q})]_{c_i}$. Then we have 
	the system $\underline{aT^c} = a_1T^{c_1}, \ldots,a_tT^{c_t}$ of elements of $R_A(\mathfrak{q})$. Note 
	that $\deg a_iT^{c_i} = c_i, i = 1,\ldots,t$. Then we may consider the Koszul complex 
	$K_{\bullet}(\underline{aT^c};R_M(\mathfrak{q}))$. This is a complex of graded $R_A(\mathfrak{q})$-modules. 
	It is easily seen that the degree $n$-component $[K_{\bullet}(\underline{aT^c};R_M(\mathfrak{q}))]_n$ 
	of $K_{\bullet}(\underline{aT^c};R_M(\mathfrak{q}))$ is the complex $K_{\bullet}(\au,\mathfrak{q},M;n)$ as introduced 
	in \ref{kos-2}. We write $H_i(\au,\mathfrak{q},M;n)$ for the $i$-th 
	homology of $K_{\bullet}(\au,\mathfrak{q},M;n)$ for $i \in \mathbb{Z}$.
\end{remark}

We come now to the definition of second complex.

\begin{definition} \label{def-1}
	With the previous notation we define $\mathcal{L}_{\bullet}(\au,\mathfrak{q},M;n)$ the quotient of the 
	embedding $K_{\bullet}(\au,\mathfrak{q},M;n) \to K_{\bullet}(\au;M)$. That is there is a short exact sequence 
	of complexes
	\[
	0 \to K_{\bullet}(\au,\mathfrak{q},M;n) \to K_{\bullet}(\au;M) \to \mathcal{L}_{\bullet}(\au,\mathfrak{q},M;n) 
	\to 0.
	\]
	Note that $\mathcal{L}_i(\au,\mathfrak{q},M;n) \cong \oplus_{1 \leq j_1 < \ldots < j_i \leq t}M/\mathfrak{q}^{n-c_{j_1}- \ldots- c_{j_i}}M$.
	The boundary maps are those induced by the Koszul complex. We write $L_i(\au,\mathfrak{q},M;n)$ 
	for the $i$-th homology of $\mathcal{L}_{\bullet}(\au,\mathfrak{q},M;n)$ and any $i \in \mathbb{Z}$.
\end{definition}

For a construction by mapping cones we need the following technical result. For a morphism $f: X \to Y$ we write 
$C(f)$ for the mapping cone of $f$.

\begin{lemma} \label{cone-2}
	With the previous notation let $b \in \mathfrak{q}^d$ an element. The multiplication map by $b$ induces 
	the following morphisms 
	\[
	m_b(K) : K_{\bullet}(\au,\mathfrak{q},M;n-d) \to K_{\bullet}(\au,\mathfrak{q},M;n) \text{ and } 
	m_b(\mathcal{L}) : \mathcal{L}_{\bullet}(\au,\mathfrak{q},M;n-d) \to \mathcal{L}_{\bullet}(\au,\mathfrak{q},M;n)
	\] 
	of complexes. They induce isomorphism of complexes 
	\[
	C(m_b(K)) \cong  K_{\bullet}(\au,b,\mathfrak{q},M;n) \text{ and } C(m_b(\mathcal{L})) \cong \mathcal{L}_{\bullet}(\au,b,\mathfrak{q},M;n).
	\]
\end{lemma}

\begin{proof}
	The proof follows easily by the structure of the complexes and the mapping cone construction. 
\end{proof}

We begin with a few properties of the previous complexes. 

\begin{theorem} \label{kos-3}
	Let $\au = a_1,\ldots,a_t$ denote a system of elements of $A$, $\mathfrak{q} \subset A$ an ideal and $M$ a 
	finitely generated $A$-module. Let $n \in \mathbb{N}$ denote an integer. 
	\begin{itemize}
		\item[(a)] $H_i(\au,\mathfrak{q},M;n) \cong H_i(\au_{\sigma},\mathfrak{q},M;n)$ and 
		$L_i(\au,\mathfrak{q},M;n) \cong L_i(\au_{\sigma},\mathfrak{q},M;n)$ for all $i \in \mathbb{Z}$ and any $\sigma$,
		a permutation on $t$ letters.
		\item[(b)] $\au H_i(\au,\mathfrak{q},M;n) = 0$ and $\au L_i(\au,\mathfrak{q},M;n) = 0$ for all $i \in \mathbb{Z}$.
		\item[(c)] $H_i(\au,\mathfrak{q},M;n)$ and $L_i(\au,\mathfrak{q},M;n)$ are finitely generated $A/\au A$-modules 
		for all $i \in \mathbb{Z}$.
	\end{itemize}
\end{theorem}

\begin{proof}
	The statement in (a) follows by virtue of the short exact sequence of complexes in \ref{def-1} and the long exact 
	homology sequence. Note that the homology of Koszul complexes is isomorphic under permutations. 
	
	The claim in (c) is a consequence of (b) since the homology modules $H_i(\au,\mathfrak{q},M;n)$ and $L_i(\au,\mathfrak{q},M;n)$ 
	are finitely generated $A$-modules. 
	
	For the proof of (b) we follow the mapping cone construction of \ref{cone-2} with the arguments of \ref{cone-1}. To this end 
	let $K_n = K_{\bullet}(\au,\mathfrak{q},M;n)$ and $C = C(m_b(K)_n) = K_{\bullet}(\au,b,\mathfrak{q},M;n)$. Then there is a short exact sequence 
	of complexes 
	\[
	0 \to K_n \to C \to K_{n-d}[-1] \to 0. 
	\]
	The differential $\partial_i$ on $(x,y) \in C_i = (K_{n-d})_{i-1} \oplus (K_n)_i$ is given by 
	\[
	\partial_i(x,y) = (d_{i-1}(x),d_i(y) +(-1)^{i-1}b x).
	\]
	Suppose that $\partial_i(x,y) = 0$, i.e., $d_{i-1}(x) =0$ and $d_i(y) = (-1)^i bx$. Then 
	$$(y,0) \in (K_{n-d})_i\oplus (K_n)_{i+1} =C_{i+1} \text{ because } (K_n)_i \subseteq (K_{n-d})_i$$ 
	and therefore $\partial_{i+1}((-1)^iy,0) = b(x,y)$. That is $b H_i(C) = 0$ for all $i \in \mathbb{Z}$. 
	
	In order to show the claim in (b) we use the previous argument. So let us consider 
	$K_{\bullet}(\au,\mathfrak{q},M;n) = C(m_{a_t}(K_{\bullet}(\au',\mathfrak{q},M;n)))$, where 
	$\au' = a_1,\ldots,a_{t-1}$. The previous argument shows $a_t H_i(\au,\mathfrak{q},M;n) = 0$ for all 
	$i \in \mathbb{Z}.$ By view of (a) this finishes the proof in the case of $H_i(\au,\mathfrak{q},M;n)$. 
	
	For the proof of $\au L_i(\au,\mathfrak{q},M;n) = 0$ we follow the same arguments. Instead 
	of the injection 
	$(K_n)_i \subseteq (K_{n-d})_i$ we use the surjection $(\mathcal{L}_n)_i \twoheadrightarrow (\mathcal{L}_{n-d})_i$, where 
	$\mathcal{L}_n = \mathcal{L}_{\bullet}(\au,\mathfrak{q},M;n)$. We skip the details here.
\end{proof}

\section{Euler characteristics}
Let $A$ denote a commutative ring. Let $X$ denote a complex of $A$-modules. 

\begin{definition} \label{def-2}
	Let $X : 0 \to X_n \to \ldots \to X_1 \to X_0 \to 0$ denote a bounded complex of $A$-modules. 
	Suppose that $H_i(X), i = 0, 1, \ldots, n,$ is an $A$-module of finite length. Then 
	\[
	\chi_A(X) = \sum_{i=0}^n (-1)^i \ell_A(H_i(X))
	\]
	is called the Euler characteristic of $X$.
\end{definition}

We collect a few well known facts about Euler characteristics. 

\begin{lemma} \label{char}
	Let $A$ denote a Noetherian commutative ring. 
	\begin{itemize}
		\item[(a)] Let $0 \to X' \to X \to X'' \to 0$ denote a short exact sequence of complexes such that all the homology 
		modules are of finite length. Then 
		$
		\chi_A(X) = \chi_A(X') + \chi_A(X'').
		$
		\item[(b)] Suppose $X: 0 \to X_n \to \ldots \to X_1 \to X_0 \to 0$ is a bounded complex such 
		that $X_i, i = 0,\ldots,n,$ is of finite length. Then 
		$
		\chi_A(X) = \sum_{i=0}^n (-1)^i \ell_A(X_i)
		$
	\end{itemize}
\end{lemma}

\begin{proof}
	The statement in (a) follows by the long exact cohomology sequence derived by $0 \to X' \to X \to X'' \to 0$. The 
	second statement might be proved by induction on $n$, the length of the complex $X$.
\end{proof}

As an application we get the following result about multiplicities, originally shown by 
\cite{AB} and \cite{jpS}.

\begin{proposition} \label{mult-1}
	Let $(A,\mathfrak{m})$ be a local ring and $a_1,\ldots,a_d \in \mathfrak{m}$ a system of parameters 
	for $M$, a finitely generated $A$-module. Then
	\[
	\chi_A(\au;M) = e_0(\au;M),
	\]
	where $\chi_A(\au;M) = \chi_A(K_{\bullet}(\au;M))$ and $e_0(\au;M)$ denotes the Hilbert-Samuel 
	multiplicity.  
\end{proposition}

\begin{proof} Let $\au = a_1,\ldots,a_d$ be the system of parameters and $\au A = \mathfrak{q}$. 
	We choose $c_i = 1, i = 1,\ldots,d$. Then the short 
	exact sequence of \ref{def-1} has the following form 
	\[
	0 \to K_{\bullet}(\au,\au,M;n) \to K_{\bullet}(\au;M) \to \mathcal{L}_{\bullet}(\au,\au,M;n) \to 0.
	\]
	All of the three complexes have homology modules of finite length and therefore 
	$\chi_A(\au;M) = \chi_A(K_{\bullet}(\au,\au,M;n))+ \chi_A(\mathcal{L}_{\bullet}(\au,\au,M;n))$ for all $n \in \mathbb{N}$.
	
	First we show that $\chi_A(K_{\bullet}(\au,\au,M;n)) = 0$ for all $n \gg 0$. To this end recall that 
	$H_i(\au,\au,M;n) = [H_i(\underline{aT};R_M(\au))]_n$. We know that $\underline{aT} (H_i(\underline{aT};R_M(\au))) = 0$ for all 
	$i = 0, \ldots,d$. Therefore $H_i(\underline{aT};R_M(\au))$ is a finitely generated module over 
	$R_A(\au)/\underline{aT} R_A(\au) = A/\au A$. This implies that $[H_i(\underline{aT};R_M(\au))]_n = H_i(\au,\au,M;n) = 0$ 
	for all $n \gg 0$. That is $\chi_A(K_{\bullet}(\au,\au,M;n)) = 0$ for all $n \gg 0$. 
	
	By view of Lemma \ref{char} (b) we get $\chi_A(\mathcal{L}_{\bullet}(\au,\au,M;n)) = \sum_{i=0}^d (-1)^i\binom{d}{i}
	\ell_A(M/\au^{n-i}M)$. For $n \gg 0$ the length $\ell_A(M/\au^nM)$ is given by the Hilbert polynomial 
	$e_0(\au;M) \binom{d+n}{d} + \ldots + e_d(\au;M)$. Therefore, it follows that $\chi_A(\mathcal{L}_{\bullet}(\au,\au,M;n)) =
	e_0(\au;M)$. This completes the argument.
\end{proof}

The more general situation of a system of parameters $\au = a_1,\ldots, a_d$ of a finitely generated $A$-module 
$M$ of a local ring $(A,\mathfrak{m})$ and an ideal $\mathfrak{q} \supset \au$ with $a_i \in \mathfrak{q}^{c_i},
i = 1,\ldots,d$ is investigated in the following.

\begin{proposition} \label{mult-2}
	With the previous notation we have the equality 
	\[
	e_0(\au;M) = c_1 \cdot \ldots \cdot c_d \cdot e_0(\mathfrak{q};M) + \chi_A(K_{\bullet}(\au,\mathfrak{q},M;n))
	\]
	for all $n \gg 0$. In particular, for all $n \gg 0$ the Euler characteristic $\chi_A(K_{\bullet}(\au,\mathfrak{q},M;n))$ 
	is a constant.
\end{proposition}

\begin{proof}
	The proof follows by the inspection of the short exact sequence of complexes 
	\[
	0 \to K_{\bullet}(\au,\mathfrak{q},M;n) \to K_{\bullet}(\au;M) \to \mathcal{L}_{\bullet}(\au,\mathfrak{q},M;n) \to 0
	\]
	of \ref{def-1}. By view of Proposition \ref{mult-1} we have $e_0(\au;M)$ for the Euler characteristic 
	of the complex in the middle. For the Euler characteristic on the right we get (see \ref{char} (b))
	\[
	\chi_A(\mathcal{L}_{\bullet}(\au,\mathfrak{q},M;n)) = \sum_{i=0}^{d} (-1)^i \sum_{1 \leq j_1 < \ldots < j_i \leq d}
	\ell_A(M/\mathfrak{q}^{n-c_{j_1}-\ldots- c_{j_i}}M),
	\]
	which gives the first summand in the above formula (see also \cite{BS} for the details in the case of $M = A$).
	This finally proves the claim.
\end{proof}

For several reasons it would be interesting to have an answer to the following problem.

\begin{problem} \label{prob}
	With the notation of Proposition \ref{mult-2} it would be of some interest to give 
	an interpretation of $\chi_A(\au,\mathfrak{q},M) := \chi_A(K_{\bullet}(\au,\mathfrak{q},M;n))$ for 
	large $n \gg 0$ independently of $n$. By a slight 
	modification of an argument given in \cite{BS} it follows that $\chi_A(\au,\mathfrak{q},M) \geq 0$. 
\end{problem} 

First we shall use the previous results in order to prove a few formulas for the multiplicity. 
The following Theorem provides a simplified proof of some of the main results of Auslander and Buchsbaum (see \cite{AB}), originally proved by the use of spectral sequences.

\begin{theorem} \label{mult-3}
	Let $(A,\mathfrak{m})$ denote a local ring with $\au = a_1,\ldots,a_d, d = \dim_A M,$ a system 
	of parameters for a finitely generated $A$-module $M$. 
	\begin{itemize}
		\item[(a)] $e_0(\au;M) = e_0(\au';M/aM) - e_0(\au'; 0:_M a)$, where $a = a_1$ and $\au' = a_2,\ldots,a_d$.
		\item[(b)] $e_0(\au;M) = e_0(a,\au';M) + e_0(b,\au';M)$, where $a_1 = a\cdot b$ and 
		$\au' = a_2,\ldots,a_d$.
		\item[(c)] $e_0(\au^{\underline{n}};M) = n \cdot e_0(\au;M)$, where $\au^{\underline{n}} = a_1^{n_1}, \ldots,a_d^{n_d}$ for $\underline{n} = (n_1,\ldots,n_d) \in \mathbb{N}^d$ and $n = n_1 \cdot \ldots \cdot n_d$.  
	\end{itemize}
\end{theorem}

\begin{proof}
	We start with the comparison of the Koszul complexes $K_{\bullet}(b;M)$ and $K_{\bullet}(ab;M)$. This leads to the following commutative diagram with exact rows
	\[
	\xymatrix{
		&  &    & 0 \ar[d]   & 0 \ar[d] &  &  &  &  \\ 
		&  & 0  \ar[r]  & M \ar[d]_{b}  \ar@{=}[r] & M  \ar[r] \ar[d]^{ab} & 0 &  &  \\ 
		& 0  \ar[r] & 0:_M a \ar[r]  & M \ar[d]  \ar[r]^{a} & M \ar[d]  \ar[r] & M/aM \ar[r] & 0 \\ 
		&     &    & 0 & 0 &  &  & 
		}
	\]
	where the columns are the Koszul complexes. That is, we have an exact sequence of complexes 
		\[
		0 \to 0:_M a \to K_{\bullet}(b;M) \to K_{\bullet}(ab;M) \to M/aM \to 0.
		\]
	We tensor now by $K_{\bullet}(\au';A)$, a bounded 
	complex of free $A$-modules. By the definition it induces an exact sequence of complexes
	\[
	0 \to K_{\bullet}(\au';0:_Ma) \to K_{\bullet}(b,\au';M) \to K_{\bullet}(ab,\au';M) 
	\to K_{\bullet}(\au';M/aM) \to 0.
	\]
	Now we inspect the previous sequence in the case $b =1$. Then it follows that $ K_{\bullet}(1,\au';M)$ 
	is exact. By virtue of the Euler characteristic and by \ref{mult-1}, the statement in (a) follows. 
	
	Next we look at the case of a general $b$. With the previous result (a) it implies that 
	\[
	e_0(ab,\au';M) = e_0(a,\au';M) +e_0(b,\au';M)
	\]
	which proves (b).
	
	By induction and permutability of the sequence this yields the statement in (c).
\end{proof}

It is noteworthy to say that $e_0(\au';0:_Ma) \not= 0$ if and only if $\dim_A 0:_Ma = \dim_A M -1$.

\section{On an equality}
As before let $(A,\mathfrak{m})$ denote a local ring with $M$ a finitely generated $A$-module. 
Let $\mathfrak{q}$ denote an ideal of $A$ such that $\ell_A(M/\mathfrak{q}M) < \infty$. 
For a system of parameters $\au = a_1,\ldots,a_d, d = \dim_A M,$ of $M$ suppose that 
$a_i \in \mathfrak{q}^{c_i} \setminus \mathfrak{q}^{c_i+1}, i = 1,\ldots,d$. Besides of the 
Rees module $R_M(\mathfrak{q})$ we need the following.

\begin{definition} \label{eq-1}
	With the previous notation we investigate the ring $\mathfrak{R} = A[a_1T^{c_1},\ldots,a_dT^{c_d}] \subset A[T]$. 
	Note that $\mathfrak{R}$ is a graded $A$-algebra. Because of $a_i \mathfrak{q}^n \subseteq \mathfrak{q}^{n+c_i}, 
	i = 1,\ldots,d$ it follows that the Rees module $R_M(\mathfrak{q})$ is an $\mathfrak{R}$-module. 
\end{definition}

In the following we shall explore when $R_M(\mathfrak{q})$ is a finitely generated $\mathfrak{R}$-module. 

\begin{lemma} \label{eq-2}
	With the previous notation the following conditions are equivalent:
	\begin{itemize}
		\item[(i)] $R_M(\mathfrak{q})$ is a finitely generated $\mathfrak{R}$-module.
		\item[(ii)] There is a positive integer $k$ such that $\mathfrak{q}^n M = \sum_{i=1}^{d} a_i\mathfrak{q}^{n-c_i}M$ for all $n > k$. 
		\item[(iii)] The initial forms $a_1^{\star}, \ldots, a_d^{\star}$ are a system of parameters 
		of $G_M(\mathfrak{q})$.
	\end{itemize}
\end{lemma}

\begin{proof}
	First we prove (i) $\Longrightarrow$ (ii). If $R_M(\mathfrak{q})$ is finitely generated over $\mathfrak{R}$, 
	then $$R_M(\mathfrak{q}) = \mathfrak{R}(m_1T^{d_1},\ldots , m_rT^{d_r}) \text{ for } m_iT^{d_i} \in R_M(\mathfrak{q})_{d_i}.$$ 
	Let $k = \max \{d_1,\ldots , d_r\}$ and $n > k$. Let $m \in \mathfrak{q}^nM$ and therefore 
	\[
	m T^n = \sum_{j=1}^{r} (\sum_{\underline{\alpha}_j} r_{\underline{\alpha}}^{(j)}(a_1T^{c_1})^{\alpha_{j,1}} 
	\cdots (a_dT^{c_d})^{\alpha_{j,d}})\,  m_j T^{d_j} \text{ with } r_{\underline{\alpha}}^{(j)} \in A
	\] 
	and $d_j + \sum c_i \alpha_{j,i} \geq n$. The last inequality implies 
	\[
	\sum c_i \alpha_{j,i} \geq n - d_j \geq n-k > 0.
	\]
	Therefore for each $j$ there is an $i$ such that $\alpha_{j,i} > 0$ and 
	$mT^n \in \sum_{i=1}^{d} a_i \mathfrak{q}^{n-c_i}MT^n$, 
	which proves the statement in (ii). 
	
	For the converse we claim $R_M(\mathfrak{q}) = \mathfrak{R}(M,\mathfrak{q}MT, \ldots, \mathfrak{q}^kM T^k).$ 
	Let $m T^n$ with $m \in \mathfrak{q}^nM$. If $n \leq k$ then clearly $mT^n \in \mathfrak{R}(M,\mathfrak{q}MT, \ldots, \mathfrak{q}^kM T^k)$. If $n > k$, then $m = \sum_{i=1}^{d} a_i m_i, m_i \in \mathfrak{q}^{n-c_i}M$.
	By induction $m_i T^{n-c_i} \in \mathfrak{R}(M,\mathfrak{q}MT, \ldots, \mathfrak{q}^kM T^k)$ and therefore 
	$mT^n = \sum_{i=1}^{d} a_iT^{c_i}m_i T^{n-c_i} \in\mathfrak{R}(M,\mathfrak{q}MT, \ldots, \mathfrak{q}^kM T^k)$, 
	which finishes the argument.
	
	For the equivalence of (ii) and (iii) recall that $a_1^{\star}, \ldots, a_d^{\star}$ is a homogeneous 
	system of parameters if and only if $[G_M(\mathfrak{q})/(a_1^{\star}, \ldots, a_d^{\star})G_M(\mathfrak{q})]_n = 0$ for all $n > k$. This is equivalent to 
	\[
	\mathfrak{q}^nM = \sum_{i=1}^{d} a_i \mathfrak{q}^{n-c_i}M + \mathfrak{q}^{n+1}M \text{ for all } n > k.
	\] 
	By Nakayama Lemma this is equivalent to the condition in (ii). 
\end{proof}

For the following we define $c = c_1 \cdot \ldots \cdot c_d$ and $e_i = c/c_i, i = 1,\ldots,d$. Then 
$\au^{\underline{e}} = a_1^{e_1}, \ldots,a_d^{e_d}$ is a system of parameters of $M$ and 
$a_i^{e_i} \in \mathfrak{q}^c$. With these notation we get the following commutative diagram
\[
\begin{array}{ccc}
R_M(\mathfrak{q}^c) & \subset & R_M(\mathfrak{q})\\
 \cup & & \cup \\
 \mathfrak{S} & \subset & \mathfrak{R},
\end{array}
\]
where $\mathfrak{S} = A[a_1^{e_1}T^c, \ldots, a_d^{e_d}T^c]$. It is easily seen that $\mathfrak{S} \subset  \mathfrak{R}$ is a finitely generated extension since it is integral. Note that $a_i^{e_i}T^c = (a_iT^{c_i})^{e_i}, 
i = 1,\ldots,d$.

\begin{corollary} \label{eq-3}
	With the previous notation the following conditions are equivalent:
	\begin{itemize}
		\item[(i)] The initial forms $a_1^{\star}, \ldots, a_d^{\star}$ are a system of parameters 
		of $G_M(\mathfrak{q})$.
		\item[(ii)] There is an integer $k > 0$ such that $\mathfrak{q}^{nc} M = \sum_{i=1}^{d} a_i^{e_i} \mathfrak{q}^{nc-c}M$ for all $n > k$.
		\item[(iii)] The initial forms $(a_1^{e_1})^\star,\ldots,(a_d^{e_d})^\star$ in $G_A(\mathfrak{q}^c)$ are 
		a system of parameters in $G_M(\mathfrak{q}^c)$.
	\end{itemize}
\end{corollary}

\begin{proof}
	We have the isomorphism $R_M(\mathfrak{q}^c) \cong M[\mathfrak{q}^cT^c] =: \mathfrak{M}$ and 
	$$
	R_M(\mathfrak{q}) \cong \mathfrak{M} \oplus \mathfrak{q}T \mathfrak{M} \oplus \ldots 
	\oplus  \mathfrak{q}^{c-1}T^{c-1} \mathfrak{M}.
	$$ 
	Whence $\mathfrak{S} \subset R_M(\mathfrak{q}^c)$ is a finitely generated extension if and only if 
	$\mathfrak{S} \subset R_M(\mathfrak{q})$ is a finitely generated extension. Then Artin-Rees 
	lemma yields the equivalence of (i) and (ii) by view of \ref{eq-2}.
	
	The equivalence of (ii) and (iii) follows by Lemma \ref{eq-2}. 
\end{proof}

\begin{corollary} \label{eq-4}
	With the previous notation suppose that $a_1^{\star}, \ldots, a_d^{\star}$ forms a system of parameters 
	in $G_M(\mathfrak{q})$. Then $e_0(\au;M) = c_1 \cdot \ldots \cdot c_d \cdot e_0(\mathfrak{q};M)$ and therefore   $\chi_A(K_{\bullet}(\au,\mathfrak{q},M;n)) = 0$ for $n \gg 0$.
\end{corollary}

\begin{proof}
	By view of \ref{eq-3} we get $(\mathfrak{q}^c)^n M = \sum_{i=1}^{d} a_i^{e_i} (\mathfrak{q}^c)^{n-1} M$ and 
	$({\mathfrak{q}^c})^l M \subset \au^{\underline{e}}M \subset \mathfrak{q}^cM$ for some $l \in \mathbb{N}$. 
	Moreover it follows that $(\mathfrak{q}^c)^{n+k}M \subseteq (\au^{\underline{e}})^k M \subseteq (\mathfrak{q}^c)^k M$ for all $k \geq l$. Then 
	\[
	\ell_A(M/(\mathfrak{q}^{c})^{n+k}M) \geq \ell_A(M/(\au^{\underline{e}})^k M) \geq 
	\ell_A(M/(\mathfrak{q}^c)^k M),
	\]
	which implies that $e_0(\mathfrak{q}^c;M) = e_0(\au^{\underline{e}};M)$. Because of $e_0(\mathfrak{q}^c;M) = 
	c^d\cdot e_0(\mathfrak{q};M)$ as easily seen and $e_0(\au^{\underline{e}};M) = e_1 \cdot \ldots \cdot e_d\cdot e_0(\au;M)$ (see \ref{mult-3}). This finishes the proof. 
\end{proof}

The previous result is a generalization of \cite[Theorem 5.1]{BS} to the situation of finitely 
generated $A$-modules. In the case of a formally equidimensional ring the converse is also true 
(see \cite[Theorem 5.2]{BS}).

\section{The subregular case}
As a consequence of the definition of $\mathcal{L}_{\bullet}(\au,\mathfrak{q},M;n)$ we get the following 
equality 
\[
c_1\cdot \ldots \cdot c_d \cdot e_0(\mathfrak{q};M) = \sum_{i=1}^{d} (-1)^i \ell_A(L_i(\au,\mathfrak{q},M;n)) \text{ for all } n \gg 0.
\]
We have $L_0(\au,\mathfrak{q},M;n) \cong M/(\au,\mathfrak{q}^n)M$. The other homology modules are 
difficult to describe. For the vanishing of some of them in relation to the existence of $G_M(\mathfrak{q})$-regular 
sequences we refer to \cite{K}. 

In the previous section we have shown that $e_0(\au;M) = c \cdot e_0(\mathfrak{q};M)$ provided 
$a_1^{\star}, \ldots, a_d^{\star}$ forms a system of parameters of $G_M(\mathfrak{q})$. In the 
next, we consider the case that $(a_1^{\star}, \ldots, a_d^{\star})G_A(\mathfrak{q})$ contains 
a $G_M(\mathfrak{q})$-regular sequence of length $d-1$. 

We start with the behavior of $\chi(\au,\mathfrak{q},M;n)$ by passing to a certain element. 
To this end, let $\au' = a_2,\ldots,a_d, a = a_1$ and $c_1 =f$. 

\begin{lemma} \label{reg-1}
	Suppose that $a^{\star}$ is $G_M(\mathfrak{q})$ regular. Then  $\chi(\au,\mathfrak{q},M;n) = \chi(\au',\mathfrak{q},M/aM;n)$ for all $n \in \mathbb{Z}$. 
\end{lemma}

\begin{proof}
	If $a^{\star}$ is $G_M(\mathfrak{q})$-regular, then $0:_M a = 0$ and 
	$\mathfrak{q}^{n}M :_M a = \mathfrak{q}^{n-f}M$ for all $n \geq 0$. So it follows that 
	$0:_{R_M(\mathfrak{q})} aT^f = 0$ and $R_M(\mathfrak{q})/(aT^f)R_M(\mathfrak{q}) \cong R_{M/aM}(\mathfrak{q})$. 
	
	Moreover, there is the following diagram with exact rows
	\[
	\xymatrix{
		&  & 0 \ar[d]   & 0 \ar[d] &  &  &    \\ 
		& 0  \ar[r]  & R_M(\mathfrak{q})(-f) \ar@{=}[d]  \ar@{=}[r] & R_M(\mathfrak{q})(-f)  \ar[r] \ar[d]^{aT^f} & 0 &  &  \\ 
		& 0 \ar[r]  & R_M(\mathfrak{q})(-f) \ar[d]  \ar[r]^{aT^f} & R_M(\mathfrak{q}) \ar[d]  \ar[r] & R_{M/aM}(\mathfrak{q}) \ar[r] & 0 & \\ 
		&    & 0 & 0 &  &  & 
	}
	\]
	where the columns are the Koszul complexes $K_{\bullet}(1;R_M(\mathfrak{q}))(-f)$ and 
	$K_{\bullet}(aT^f;R_M(\mathfrak{q}))$ resp. That is, we have a short exact sequence of complexes 
	\[
	0 \to K_{\bullet}(1;R_M(\mathfrak{q}))(-f) \stackrel{aT^f}{\longrightarrow} 
	K_{\bullet}(aT^f;R_M(\mathfrak{q})) \to R_{M/aM}(\mathfrak{q}) \to 0.
	\]
	By tensoring with $K_{\bullet}(\underline{a'T^{c'}}, R_A(\mathfrak{q}))$ it provides a
	short exact sequence of complexes
	\[
	0 \to K_{\bullet}(1,\underline{a'T^{c'}},R_M(\mathfrak{q}))(-f) \stackrel{aT^f}{\longrightarrow} K_{\bullet}(\underline{a T^c};R_M(\mathfrak{q})) \to K_{\bullet}(\underline{a'T^{c'}};R_{M/aM}(\mathfrak{q}))\to 0.
	\]
	Since the first Koszul complex in the previous sequence is exact (see \ref{cone-1}), 
	the claim follows by the definition.
\end{proof}

The following result is a particular case of \cite[VIII, Lemma 3]{ZS} 
resp. to \cite[Lemma 1.6]{RV},  a generalization to filtered modules.

\begin{lemma} \label{reg-2}
	Let $\mathfrak{q} \subset A$ denote an ideal such that $\ell_A(M/\mathfrak{q}M) < \infty$.
	\begin{itemize}
		\item[(a)] Suppose that $a^{\star}$ is $G_M(\mathfrak{q})$ regular. Then
		$f \cdot e_0(\mathfrak{q};M) = e_0(\mathfrak{q};M/aM)$.
		\item[(b)] Let $\dim_A M = 1$ and $a \in \mathfrak{q}^f$ be a parameter of $M$. Then 
		$$
		f\cdot e_0(\mathfrak{q};M) = \ell_A(M/aM) - \ell_A(\mathfrak{q}^nM :_M a/\mathfrak{q}^{n-f}M)
		$$ 
		for all $n \gg 0$. In particular $\ell_A(\mathfrak{q}^nM :_M a/\mathfrak{q}^{n-f}M)$ is a constant 
		for all $n \gg 0$.
	\end{itemize}
\end{lemma}

\begin{proof}
	The statements follow by counting the length in the exact sequence
	\[
	0 \to \mathfrak{q}^n M :_M a/\mathfrak{q}^{n-f} M \to M/\mathfrak{q}^{n-f} M \stackrel{a}{\to} 
	M/\mathfrak{q}^n M \to M/(a,\mathfrak{q}^n) M \to 0
	\] 
	for all $n \in \mathbb{N}$. If $a^{\star}$ is $G_M(\mathfrak{q})$ regular, then 
	$\mathfrak{q}^n M :_M a = \mathfrak{q}^{n-f} M$ for all $n \in \mathbb{N}$. If 
	$\dim_A M = 1$ and $a \in \mathfrak{q}^f$ is a parameter of $M$, then 
	$\ell_A(M/\mathfrak{q}^nM) - \ell_A(M/\mathfrak{q}^{n-f}M) = f \cdot e_0(\mathfrak{q};M)$
	for all $ n \gg 0$. Moreover $\mathfrak{q}^n M \subset aM$ for all $n \gg 0$. For details see \cite[VIII, Lemma 3]{ZS} resp. \cite[Lemma 1.6]{RV}.
%
\end{proof}

Now we are prepared for the main situation of this section. 

\begin{corollary} \label{reg-3}
	With the previous notation assume that $a_1^{\star},\ldots,a_{d-1}^{\star}$ is a $G_M(\mathfrak{q})$-regular 
	sequence. Then 
	\[
	c_1\cdot \ldots \cdot c_d \cdot e_0(\mathfrak{q};M) = \ell_A(M/\au M) - \ell_A((\au',\mathfrak{q}^n)M:_M a_d/ 
	(\au',\mathfrak{q}^{n-c_d})M)
	\]
	for all $n \gg 0$ where $\au' = a_1,\ldots,a_{d-1}$.
\end{corollary}

\begin{proof}
	By an iterative application of \ref{reg-2} (a) it follows $c_1\cdot \ldots\cdot c_{d-1} \cdot e_0(\mathfrak{q};M) 
	= e_0(\mathfrak{q};M/\au' M)$. By view of \ref{reg-2} (b) it implies $c_d \cdot e_0(\mathfrak{q};M/\au' M) 
	= \ell_A(M/\au M) - \ell_A((\au',\mathfrak{q}^n)M:_M a_d/ 
	(\au',\mathfrak{q}^{n-c_d})M)$. Putting both together yields the claim.
\end{proof}

It is a problem to give an interpretation of the constant $\ell_A((\au',\mathfrak{q}^n)M:_M a_d/ 
(\au',\mathfrak{q}^{n-c_d})M)$ in intrinsic data of the module $M$. As a partial result in this direction we get the 
following.

\begin{lemma} \label{reg-4}
	With the notation and assumption of \ref{reg-3} we have the following inequality
	\[
	\ell_A(M/\au M) \geq c \cdot e_0(\mathfrak{q};M) + \mathfrak{x},
	\]
	where $\mathfrak{x}=  \ell_A([\Ext^{d-1}_{G_A(\mathfrak{q})}(G_A(\mathfrak{q})/(\underline{a}^{\star}),G_M(\mathfrak{q}))]_{n-s-1})$
	is a constant for $n \gg 0$ and  $s$ denotes $s = c_1+\ldots +c_d$.
\end{lemma}

\begin{proof}
	There is an injection of 
	\[
	\mathfrak{X} := [(\au'^{\star})G_M(\mathfrak{q}) : a_d^{\star}/(\au'^{\star})G_M(\mathfrak{q})]_{n-c_d-1} \cong 
	((\au',\mathfrak{q}^n)M:_M a_d) \cap (\au',\mathfrak{q}^{n-c_d-1})M)/ (\au',\mathfrak{q}^{n-c_d})M
	\]
	into $(\au',\mathfrak{q}^n)M:_M a_d/ (\au',\mathfrak{q}^{n-c_d})M$. Moreover, there is an isomorphism 
	$$ 
	\mathfrak{X} \cong [\Hom_{G_A(\mathfrak{q})}(G_A(\mathfrak{q})/(\underline{a}^{\star}), G_M(\mathfrak{q})/
	(\au'^{\star})G_M(\mathfrak{q}))]_{n-c_d-1}.
	$$
	Since $\au'^{\star}$ is a $G_M(\mathfrak{q})$-regular sequence of length $d-1$, it follows that 
	\[
	\mathfrak{X} \cong [\Ext^{d-1}_{G_A(\mathfrak{q})}(G_A(\mathfrak{q})/(\underline{a}^{\star}),G_M(\mathfrak{q}))]_{n-s-1},
	\]
	which proves the claim.
\end{proof}

We conclude with a geometric application on the local Bezout numbers. 

\begin{example} \label{reg-5}
	Let $C = V(F), D = V(G) \subset \mathbb{P}^2_{\Bbbk}$ be two curves in the projective plane without a common 
	component. Let $\mu(P;C,D)$ denote the local intersection multiplicity of $P$ in $C \cap D$. In the case 
	of $P$ is the origin, it follows that $\mu(P;C,D) = \ell_A(A/(f,g))$, where $A = \Bbbk[x,y]_{(x,y)}$ and $f,g$ denote 
	the equations in $A$. Since $C,D$ have no component in common, $\{f,g\}$ forms as system of parameters in $A$. 
	Let $\mathfrak{m}$ denote the maximal ideal of $A$. Then $B := \Bbbk[X,Y] = G_A(\mathfrak{m})$ and $1 = e_0(\mathfrak{m};A)$. 
	We distinguish two cases:
	\begin{itemize}
		\item[1.] $C$ and $D$ intersect transversally in the origin. Then $f^{\star},g^{\star}$ form a 
		homogeneous system of parameters in $B$ and therefore
		\[
		\ell_A(A/(f,g)) = c \cdot d,
		\]
		where $c,d$ denote the initial degree of $f,g$ resp.
		\item[2.] Suppose that $C$ and $D$ do not intersect transversally. Then 
		\[
		\ell_A(A/(f,g)) \geq c\cdot d +t,
		\]
		where $t$ denotes the number of common tangents of $f$ and $g$ at the origin counted with multiplicities. 
	\end{itemize}
	\begin{proof}
		First note that $A$ is a Cohen-Macaulay ring and therefore $e_0(f,g;A) = \ell_A(A/(f,g))$. Then the 
		equality in the first case is a consequence of \ref{eq-4}. To this end note that $f^{\star},g^{\star}$ 
		forms a system of parameters in $B$ provided $C$ and $D$ intersect transversally in the origin. 
		
		For the second case we use \ref{reg-4}. To this end we have to show that $\mathfrak{x} =t$. We 
		put $\mathfrak{Y} = f^{\star}B:_B g^{\star}/f^{\star}B$. Since $f^{\star},g^{\star}$ are not relatively 
		prime, we write $f^{\star} = h\cdot f', g^{\star} = h \cdot g'$ with homogeneous polynomials $f',g',h \in B$, 
		where $f',g'$ are relatively prime. Then 
		\[
		\mathfrak{Y} = f'hB:_B g'h/f'hB \cong f'B/f'hB \cong B/hB[-\deg f']
		\]
		and $\dim_{\Bbbk} \mathfrak{Y}_n = \deg h$ for all $n \gg 0$. Since $\deg h$ is the number of common tangents 
		counted with multiplicities, this confirms the second case.
	\end{proof}	
\end{example}

The second case was also proved by Byd\u{z}ovsk\'y (see \cite{By}). Note that this is an 
improvement of the corresponding result in \cite{Fi} where it is shown that $\ell_A(A/(f,g)) \geq c\cdot d +1$ 
in case there is a common tangent. 

A further discussion of the difference $\ell_A(A/(f,g)) - c\cdot d - t \geq 0$ is given in \cite{BS}. 
There is also another approach by blowing-ups. 

{\bf Acknowledgement:} The authors are grateful to the reviewer for bibliographical comments 
and suggestions.

\end{document}